\documentclass[12pt,a4paper]{amsart}
\usepackage{geometry}
\usepackage{amsmath,amssymb,amsfonts,amscd,amsthm,wasysym,color,enumitem,indentfirst,graphicx,booktabs}
\usepackage[bookmarksnumbered,colorlinks,linktocpage,plainpages]{hyperref}
\hypersetup{pdfstartview={FitH}}

\numberwithin{equation}{section}
\newtheorem{theorem}{Theorem}[section]

\newtheorem{proposition}[theorem]{Proposition}

\theoremstyle{definition}

\newtheorem{definition}{Definition}

\newtheorem{remark}{Remark}

\allowdisplaybreaks[4]

\makeatletter
\@namedef{subjclassname@2020}{\textup{2020} Mathematics Subject Classification}
\makeatother

\begin{document}
\title[Removable singularities of plurisubharmonic functions]{Complete pluripolar sets and removable singularities of plurisubharmonic functions}

\author{Xieping Wang}
\address{CAS Wu Wen-Tsun Key Laboratory of Mathematics and School of Mathematical Sciences, University of Science and Technology of China, Hefei 230026, Anhui, People's Republic of China}
\email{xpwang008@ustc.edu.cn}

\thanks{The author was partially supported by NSFC grants 12001513 and 12371083, and the Fundamental Research Funds for the Central Universities.}

\subjclass[2020]{32D15, 32D20, 32U30, 32U05}
\keywords{Removable singularities, plurisubharmonic functions, complete pluripolar sets}

\dedicatory{To Li and Rui'an}

\begin{abstract}
Inspired by Chen--Wu--Wang (Math. Ann. 362: 305--319, 2015), we prove a Hartogs type extension theorem for plurisubharmonic functions across a compact complete pluripolar set, which is complementary to a classical theorem of Shiffman.
\end{abstract}
\maketitle

\section{Introduction}
Let $\Omega$ be a domain in $\mathbb C^n$, $n\geq 2$, and let $K$ be a compact subset of $\Omega$ such that  $\Omega\!\setminus\!K$ is connected. The famous Hartogs extension theorem asserts that every holomorphic function on $\Omega\!\setminus\!K$ extends holomorphically to the whole domain $\Omega$. An analogue for pluriharmonic functions is also valid, as recently discovered by Chen in \cite{Chen} (see also \cite{Wang}). When it comes to  plurisubharmonic (psh for short) functions the situation, however, is quite different; for instance, every bounded domain in $\mathbb C^n$ with smooth boundary is the domain of existence of a psh function (see \cite{Bed-Tay_subext}). Thus it is meaningful to prove the following

\begin{theorem}\label{thm:Hartogs_PSH-domain}
Let $\Omega$ be a domain in ${\mathbb C}^n$, $n\geq 2$, and let $K$ be a compact complete pluripolar subset of $\Omega$. Then every psh function on $\Omega\!\setminus\!K$ admits a unique psh extension to $\Omega$.
\end{theorem}

This somewhat surprising result is complementary to a classical theorem of Shiffman \cite{Shiffman72} that every psh function on a domain in $\mathbb C^n$ extends plurisubharmonically across a closed set of Hausdorff $(2n-2)$-measure zero. To explain this, we only need to recall a standard fact in pluripotential theory that pluripolar sets in $\mathbb C^n$ have Hausdorff dimension at most $2n-2$, and  mention the existence of compact complete pluripolar sets with positive (finite or infinite) Hausdorff $(2n-2)$-measure in $\mathbb C^n$; see Section \ref{sect:Complete pluripolar sets} for precise references. It is also worth noting that Theorem \ref{thm:Hartogs_PSH-domain} is global in nature,  whereas  Shiffman's theorem (as well as  many other related results in the literature; see \cite{Siu75, Harvey_CPAM, CWW}, etc.) is local in nature; and the conclusion of Theorem \ref{thm:Hartogs_PSH-domain} does not hold in general when $n=1$ or when the compactness assumption of the complete pluripolar set $K$ is dropped, as easily seen by considering functions of the form $-\log|f|$ on $\Omega\!\setminus\!f^{-1}(0)$,  with $f$ being holomorphic and not identically vanishing on $\Omega$.

The special case of Theorem \ref{thm:Hartogs_PSH-domain} where $\Omega$ is the unit polydisc $\Delta^n\subset\mathbb C^n$ is already contained in the beautiful work of Chen--Wu--Wang \cite{CWW}, who dealt with the more general case of $K$ being a {\it closed} complete pluripolar subset of $\Delta^n$ under certain reasonable conditions. Chen--Wu--Wang proved their result by using an Ohsawa--Takegoshi type extension theorem for a single point in bounded {\it complete K\"ahler} domains, which is also one of the main results in \cite{CWW}  and seems to be highly nontrivial due to its connection with an open problem posed by Ohsawa in \cite{Ohsawa95}. We will prove Theorem \ref{thm:Hartogs_PSH-domain} by combining this powerful result with an idea of Shiffman. As in the proof of Shiffman's theorem mentioned earlier, the point here is also to show that every psh function on  $\Omega\!\setminus\!K$ is actually locally bounded above near $K$, for then the desired plurisubharmonic extendibility follows immediately from a well-known result of Lelong (see also \cite{CWW} for a related result). 

One might naturally ask whether Theorem \ref{thm:Hartogs_PSH-domain} still holds  true when $\mathbb C^n$ is replaced by a generic Stein manifold of dimension $n\geq 2$. Since  it is not clear to us at this moment whether the Ohsawa--Takegoshi type extension theorem by Chen--Wu--Wang applies to this more general case, we instead use other techniques, namely the recently proved Hartogs extension theorem for pluriharmonic functions in \cite{Wang} and the Skoda--El Mir extension theorem for closed positive currents, to prove the following

\begin{theorem}\label{thm:Hartogs_PSH-manifold}
Let $X$ be a Stein manifold of dimension $n\geq 2$. Suppose $\Omega$ is a domain in $X$ such that $H^1(\Omega,\, \mathcal{O})=0$ and $H^2(\Omega,\, \mathbb R)=0$,  and $K$ is a compact complete pluripolar subset of $\Omega$. Then every psh function on $\Omega\!\setminus\!K$ admits a unique psh extension to $\Omega$.
\end{theorem}

\begin{remark}
Since $K$ is holomorphically convex in $X$, there always exists a Stein neighborhood of $K$ contained in a given domain $\Omega\subset X$ so that the assumption $H^1(\Omega,\, \mathcal{O})=0$ is nonessential for the theorem. Also, it seems that the additional condition $H^2(\Omega,\, \mathbb R)=0$ is superfluous (and this is the case at least when $X=\mathbb C^n$, as shown by Theorem \ref{thm:Hartogs_PSH-domain}).
\end{remark}

 It would be quite interesting to know whether Theorems \ref{thm:Hartogs_PSH-domain} and \ref{thm:Hartogs_PSH-manifold} still hold  true when $K$ is only assumed to be  a compact pluripolar subset of $\Omega$.

The paper is organized as follows. In Section \ref{sect:Complete pluripolar sets} we recall a fundamental result concerning closed complete pluripolar sets in Stein manifolds. We then prove Theorems $\ref{thm:Hartogs_PSH-domain}$ and $\ref{thm:Hartogs_PSH-manifold}$ in Sections \ref{sect:Hartogs_PSH-domain} and \ref{sect:Hartogs_PSH-manifold}, respectively.

\medskip
\noindent {\bf Acknowledgements.}
The author would like to thank Professor Bo-Yong Chen for kindly explaining his joint work with Wu and Wang \cite{CWW}, and for his interest in this work. The author would also like to thank Doctors Yong-Xin Gao and Zhi Li for patiently listening to his lectures on the basic parts of pluripotential theory and related topics at the seminar they organized when all three of them were at the Institute of Mathematics, AMSS, Chinese Academy of Sciences from 2017 to 2019. Last but not least, the author is grateful to the anonymous referees for their careful reading of the paper and very accurate comments that helped improve the exposition of the paper.

\section{Complete pluripolar sets and their defining functions}\label{sect:Complete pluripolar sets}

We begin by recalling the definition of complete pluripolar sets, which are our main concern in this paper.  Let $X$ be a complex manifold and $PSH(X)$ denote the set of all psh functions on $X$.

\begin{definition}
A subset $E\subset X$ is called {\it complete pluripolar} if for every point $z\in E$ there exists a neighborhood $U$ of $z$ and a function $\varphi\in PSH(U)$ such that
  $$
  E\cap U=\varphi^{-1}(-\infty).
  $$
\end{definition}

The set of all complex subvarieties of $X$ forms a particularly important class of (closed) complete pluripolar sets, but complete pluripolar sets are much more general: for instance, the Cartesian product of finitely many (possibly distinct) Cantor type sets of logarithmic capacity zero in the complex plane is a {\it compact} complete pluripolar set in the corresponding complex Euclidean space (see, e.g., \cite{Ransfordbook}), but far from complex-analytic. One may also consult \cite{El-Mir_cmp-plolar, LMP_IUMJ, Edlund_cmp-curv, DM_cmp-plolar} and the references therein for many other nontrivial and very interesting examples of {\it compact} complete pluripolar sets in $\mathbb C^n$, especially those of minimal Hausdorff codimension (i.e., of Hausdorff codimension two).

In 1990, Col\c{t}oiu proved the following important result concerning the existence of a global defining function for a closed complete pluripolar set.

\begin{theorem}[see {\cite[Corollary 1]{Coltoiu}}]\label{thm:def-complete-polar}
Let $X$ be a Stein manifold and $E\subset X$ be a closed complete pluripolar set. Then there exists a function $\rho\in PSH(X)\cap C^{\infty}(X\!\setminus\!E)$ such that $\rho^{-1}(-\infty)=E$ and $\sqrt{-1}\partial\bar{\partial}\rho>0$ on $X\!\setminus\!E$.
\end{theorem}

As we shall see later, this result plays a fundamental role in the proofs of Theorems $\ref{thm:Hartogs_PSH-domain}$ and $\ref{thm:Hartogs_PSH-manifold}$.

\section{Proof of Theorem $\ref{thm:Hartogs_PSH-domain}$}\label{sect:Hartogs_PSH-domain}
The idea of the proof is due to Chen--Wu--Wang \cite{CWW}, which in turn was more or less inspired by the celebrated work of Demailly \cite{Demailly92}. The key ingredient here is an Ohsawa--Takegoshi type extension theorem for a single point in bounded complete K\"ahler domains in $\mathbb C^n$ (see \cite[Theorem 1.3]{CWW} for details). In order to make use of this theorem, we first need to prove the following result:

\begin{theorem}\label{thm:complete-Kahler}
Let $X$ be a Stein manifold and $E\subset X$ be a closed complete pluripolar set. Then $X\!\setminus\!E$ carries a complete K\"{a}hler metric.
\end{theorem}

The very important special case of $E\subset X$ being a complex subvariety is well-known and is due to Grauert \cite{Grauert_comp-kahler}. Note that even in this special case, $X\!\setminus\!E$ cannot be Stein unless $E$ is empty or purely one-codimensional, in view of the second Riemann extension theorem.

\begin{proof}
The result is an easy consequence of Theorem \ref{thm:def-complete-polar}. Let $\varphi$ be a $C^{\infty}$ strictly psh exhaustion function for $X$ and let
$\psi\!: X\!\setminus\!E\to \mathbb R$ be a $C^{\infty}$ function such that
   $$
   \psi=-\log(-\rho) \quad \mbox{on}\;\; U\!\setminus\!E,
   $$
where $\rho$ is a psh function as in Theorem \ref{thm:def-complete-polar} with $\sup_X\rho>-1$ and $U:=\{\rho<-1\}$. Then one can construct a $C^{\infty}$ convex, rapidly increasing function $\chi$ on $\mathbb R$ such that
   $$
   \omega:=\sqrt{-1}\partial\bar{\partial}(\chi\circ\varphi)+\sqrt{-1}\partial\bar{\partial}\psi\geq \omega_0 \quad \mbox{on}\;\; X\!\setminus\!E
   $$
for some complete K\"{a}hler metric $\omega_0$  on $X$.

We claim that $\omega$ is complete on $X\!\setminus\!E$. For this, we may assume without loss of generality that $X$ itself is connected (and so is $X\!\setminus\!E$). Suppose $\{z_j\}_{j\geq 1}$ is a bounded sequence in the metric space $(X\!\setminus\!E,\, \omega)$. Then there is a sequence of smooth curves $\{\gamma_j\}_{j\geq 1}\subset C^{\infty}([0, 1],\, X\!\setminus\!E)$ with uniformly bounded lengths with respect to $\omega$, joining each $z_j$ to a (fixed) reference point in $X\!\setminus\!\overline{U}$. Since $\omega\geq\omega_0$ on $X\!\setminus\!E$ and $\omega_0$ is complete on $X$, we may assume that the sequence $\{z_j\}_{j\geq 1}$ itself converges in $X$  by passing to a subsequence if necessary. What now remains is to show that the limit of $\{z_j\}_{j\geq 1}$ lies outside $E$. Suppose the contrary and set
   $$
   t_j:=\inf\!\big\{t\in [0,1]\!: \gamma_j([t, 1])\subset U\big\},\quad j\geq 1.
   $$
Clearly $0<t_j<1$ and $\gamma_j(t_j)\in\partial U=\{\rho=-1\}$ for all sufficiently large $j$. Observe also that
   $$
   \omega\geq \sqrt{-1}\partial\bar{\partial}\big(-\log(-\rho)\big)\geq \sqrt{-1}\partial\log(-\rho)\wedge \bar{\partial}\log(-\rho)\quad \mbox{on}\;\; U\!\setminus\!E.
   $$
It then follows that
   \begin{equation*}
      \begin{split}
      \sqrt2\,{\rm length}_{\omega}(\gamma_j)&\geq\int_{t_j}^1\big|\big(d\log(-\rho)\big)(\gamma'(t))\big|\,dt \geq\int_{t_j}^1\big(d\log(-\rho)\big)(\gamma'(t))\,dt\\
          &=\log(-\rho(z_j))\to \infty \quad {\rm as}\;\;  j\to \infty,
      \end{split}
   \end{equation*}
contradicting the boundedness of $\{{\rm length}_{\omega}(\gamma_j)\}_{j\geq 1}$. Therefore the limit of $\{z_j\}_{j\geq 1}$ lies outside $E$, and hence $\omega$ is complete on $X\!\setminus\!E$.
\end{proof}

\begin{remark}
If we denote by $d_{\omega}$ and $d_{\omega_0}$ the distance functions on $X\!\setminus\!E$ and $X$ associated to $\omega$ and $\omega_0$, respectively, and by $\chi_{U\setminus E}$  the characteristic function of $U\!\setminus\!E$, then the preceding argument actually shows the estimate
  $$
  d_{\omega}(z, z_0)\geq \max\Big\{d_{\omega_0}(z,z_0),\, \chi_{U\setminus E}\log(-\rho(z))/\sqrt{2}\Big\},\quad z\in X\!\setminus\!E,\, z_0\in X\!\setminus\!\overline{U}.
  $$
This, of course, implies the completeness of $\omega$.
\end{remark}

We are now ready to prove Theorem $\ref{thm:Hartogs_PSH-domain}$.

\begin{proof}[Proof of Theorem $\ref{thm:Hartogs_PSH-domain}$]

The uniqueness of the extension is clear, since two psh functions on $\Omega$ which coincide almost everywhere are actually equal everywhere. So it suffices to prove the  existence of the extension.

We first observe that the problem can be reduced to the case when $\Omega\supset K$ is a bounded pseudoconvex domain. To see this, let $\rho$ be a psh function on $\mathbb C^n$, continuous on $\mathbb C^n\!\setminus\!K$ and satisfying $\rho^{-1}(-\infty)=K$ (see Theorem \ref{thm:def-complete-polar}). Choose an open set $U\subset \mathbb C^n$ such that $K\subset U\subset\subset \Omega$, and set
  $$
  \widetilde{\rho}:=
  \left\{
  \begin{array}{lll}
  \!\!\max\big\{\rho,\, \inf\limits_{\partial U}\rho\big\} \quad \mbox{on}\;\; \mathbb C^n\!\setminus\!\overline{U}; \\
  \!\!\qquad  \quad  \rho\quad  \quad  \quad \;\;\,\mbox{on}\;\;\overline{U}.
  \end{array}
  \right.
  $$
Then  $\widetilde{\rho}$ is a psh function on $\mathbb C^n$ with $\widetilde{\rho}^{-1}(-\infty)=K$.
Replacing $\Omega$ by any connected component of $\{\widetilde{\rho}<\inf\limits_{\partial U}\rho\}$, we may assume in what follows that $\Omega$ is a bounded pseudoconvex domain in $\mathbb C^n$.

Let $\varphi\in PSH(\Omega\!\setminus\!K)$. To prove the plurisubharmonic extendibility of $\varphi$ across $K$, it suffices to show that every point of $K$ admits a small neighborhood on which $\varphi$ is bounded above. Observe that $\Omega\!\setminus\!K$ is a bounded complete K\"{a}hler domain, in view of Theorem \ref{thm:complete-Kahler}. We can therefore invoke \cite[Theorem 1.3]{CWW} to assign to every point $z\in\Omega\!\setminus\!K$ a holomorphic function $f_z$ on $\Omega\!\setminus\!K$ with the property that $f_z(z)=e^{\varphi(z)/2}$ and
\begin{equation}\label{ineq:norm-estimate}
  \int_{\Omega\setminus K}|f_z|^2e^{-\varphi}\leq {\rm const}_{n,\, {\rm diam}\, \Omega}.
\end{equation}
By the  Hartogs extension theorem for holomorphic functions, each such $f_z$ extends holomorphically to $\Omega$. With a slight abuse of notation, we denote the extension still by $f_z$.

To proceed with the proof we make use of a result of Shiffman. Fix an arbitrary point $z_0\in K$ and recall that being a polar subset of $\mathbb C^n\cong \mathbb R^{2n}$, $K$ has Hausdorff dimension at most $2n-2$ (see, e.g., \cite[Theorem 5.9.6]{ArGa_potential}). Then according to \cite[Lemma 2.3]{Harvey_Survey} or  \cite[Chapter 3, Lemma 4.7]{Demaillybook}), by suitably selecting affine linear coordinates for $\mathbb C^n$, we can find a polydisc $\Delta'\times\Delta''\subset \mathbb C^{n-1}\times\mathbb C$ centered at $z_0=:(z'_0, z''_0)$ such that
  $$
  (\Delta'\times\partial\Delta'')\cap K=\emptyset.
  $$
By shrinking $\Delta'$ if necessary, we further arrive at
  $$
  \big(\Delta'\times (\Delta''\!\setminus\!(1-\varepsilon)\overline{\Delta''})\big)\cap K=\emptyset
  $$
for some sufficiently small $\varepsilon>0$. The remaining argument is the same as in the last paragraph of the proof of \cite[Theorem 1.2]{CWW}. Choose $R,\, r>0$ such that $1-\varepsilon<r<R<1$ and ball $B\subset\subset\Delta'$ centered at $z'_0$. Then the Cauchy estimate and inequality \eqref{ineq:norm-estimate} imply
\begin{equation*}
\begin{split}
e^{\varphi(z)}=|f_z(z)|^2
&\leq {\rm const}_{B,\, R,\, r,\, \varepsilon}\int_{B\times (R\Delta''\setminus r\Delta'')}|f_z|^2\\
&\leq {\rm const}_{B,\, R,\, r,\, \varepsilon}\sup_{B\times (R\Delta''\setminus r\Delta'')}e^\varphi
      \int_{B\times (R\Delta''\setminus r\Delta'')}|f_z|^2e^{-\varphi}\\
&\leq C\sup_{B\times (R\Delta''\setminus r\Delta'')}e^\varphi
\end{split}
\end{equation*}
for all $z\in \big((1-\varepsilon)(B\times\Delta'')\big)\setminus K$, where $C>0$ is a constant independent of $z$. Consequently, $\varphi$ is bounded above on $(1-\varepsilon)(B\times\Delta'')\ni z_0$ outside $K$. This completes the proof.
\end{proof}

\section{Proof of Theorem $\ref{thm:Hartogs_PSH-manifold}$}\label{sect:Hartogs_PSH-manifold}

The proof involves the concept and basic properties of closed positive currents, for which we refer the reader to Demailly's book \cite{Demaillybook}. We start with the following result, which is essentially due to Sibony \cite{Sibony85}.

\begin{proposition}\label{prop:finite-mass}
Let $\Omega$ be a domain in a Stein manifold of dimension $n\geq 2$, and $K\subset \Omega$ be a compact complete pluripolar set. Then every closed positive $(p, p)$-current on $\Omega\!\setminus\!K$ with $1\leq p\leq n-1$ has finite mass near $K$.
\end{proposition}

\begin{proof}
After shrinking $\Omega$ by a technique similar to that at the beginning of the proof of Theorem \ref{thm:Hartogs_PSH-domain}, we may assume that $\Omega$ itself is a strongly pseudoconvex neighborhood of $K$ with $C^{\infty}$ boundary, which is relatively compact in the ambient manifold. Let $u_K$ denote the relative extremal function of $K$ in $\Omega$, that is
  $$
  u_K=\sup\Big\{u\in PSH(\Omega)\cap C(\Omega)\!: u<1\; {\rm on}\; \Omega,\,  u\leq 0  \;{\rm on}\; K\Big\}.
  $$
Clearly $u_K=0$ on $K$. Moreover according to \cite[Proposition 1.4]{Sibony85}\footnote{Strictly speaking, Sibony only considered the case when the ambient manifold of $\Omega$ is $\mathbb C^n$, but his argument works equally well for the more general case of Stein manifolds.}, the product $u_KT$ of $u_K$ and every closed positive $(p, p)$-current $T$ on $\Omega\!\setminus\!K$ with $1\leq p\leq n-1$ has finite mass near $K$. (This is true for all compact sets $K\subset \Omega$, regardless of the complete pluripolarity of $K$.)

It remains to show that $u_K$ is no other than the characteristic function $\chi_{\Omega\setminus K}$ of $\Omega\!\setminus\!K$, provided $K\subset \Omega$ is further assumed to be complete pluripolar. For this let $\rho$ be a negative psh function on $\Omega$, continuous on $\Omega\!\setminus\!K$ and satisfying $\rho^{-1}(-\infty)=K$, and set
  $$
  \rho_t:=\max\big\{\rho/t+1,\, 0\big\},\quad t>0.
  $$
Then $\{\rho_t\}_{t>0}$ forms a family of candidates for the supremum defining $u_K$, hence
$\rho_t\leq u_K$ for all $t>0$. Letting $t\to \infty$ yields $u_K=\chi_{\Omega\setminus K}$, as desired.
\end{proof}

Now we are in a position to prove Theorem $\ref{thm:Hartogs_PSH-manifold}$.

\begin{proof}[Proof of Theorem $\ref{thm:Hartogs_PSH-manifold}$]
As pointed out in the proof of Theorem $\ref{thm:Hartogs_PSH-domain}$, it suffices to prove the existence part of the theorem.

Given a function $\varphi\in PSH(\Omega\!\setminus\!K)$, we consider the associated closed positive $(1, 1)$-current $T:=\sqrt{-1}\partial\bar{\partial}\varphi$ on $\Omega\!\setminus\!K$. By virtue of Proposition \ref{prop:finite-mass} and the Skoda--El Mir extension theorem (see \cite[Th\'{e}or\`{e}me II.1]{El-Mir_cmp-plolar} or \cite{Sibony85, Demaillybook}), the trivial extension $\widetilde{T}$ of $T$ by zero across $K$ is a closed positive $(1, 1)$-current on $\Omega$. Since $H^1(\Omega,\, \mathcal{O})=0$ and $H^2(\Omega,\, \mathbb R)=0$, a standard argument shows that $\widetilde{T}$ admits a global potential $\widetilde{\varphi}\in PSH(\Omega)$, i.e., $\sqrt{-1}\partial\bar{\partial}\widetilde{\varphi}=\widetilde{T}$. One can then write
  $$
  \widetilde{\varphi}=\varphi+h\quad {\rm on}\ \, \Omega\!\setminus\!K
  $$
with $h$ being a pluriharmonic function on $\Omega\!\setminus\!K$, in view of Weyl's lemma. On the other hand, since $\Omega\!\setminus\!K$ is connected by the pluripolarity of $K$, the Hartogs extension theorem for pluriharmonic functions (see \cite[Theorem 1.1]{Wang}) implies that $h$ admits a pluriharmonic extension $\widetilde{h}$ to $\Omega$. It now follows that $\widetilde{\varphi}-\widetilde{h}$ is a psh function on $\Omega$ that extends $\varphi$.
\end{proof}

\end{document}